\documentclass{amsart}
\usepackage{stmaryrd}
\usepackage{xy}
\usepackage{mathrsfs}
\usepackage{amscd,amssymb,color}
\usepackage[colorlinks=true]{hyperref}

\mathchardef\varDelta="7101

\newcommand{\dgcat}{\mathsf{dgcat}}

\renewcommand{\to}{\mathchoice{\longrightarrow}{\rightarrow}{\rightarrow}{\rightarrow}}

\newcommand{\cA}{{\mathcal A}}
\newcommand{\cB}{{\mathcal B}}

\newcommand{\cD}{{\mathcal D}}

\newcommand{\cM}{{\mathcal M}}

\newcommand{\cU}{{\mathcal U}}

\newcommand{\Hmo}{{\mathsf{Hmo}_0}}

\def\quickop#1{\expandafter\DeclareMathOperator\csname
#1\endcsname{#1}}
\quickop{id}\quickop{Id}\quickop{colim}\quickop{hocolim}\quickop{op}
\quickop{co}\quickop{Ar}\quickop{sing}\quickop{Hom}\quickop{w}\quickop{Ho}
\quickop{ob}\quickop{diag}\quickop{Stab}\quickop{Cat}\quickop{Mot}\quickop{Mod}
\quickop{map}\quickop{Cone}\quickop{End}\quickop{Idem}\quickop{Perf}\quickop{Ind}
\quickop{Gap}\quickop{Shv}\quickop{Spaces}\quickop{cof}\quickop{ex}\quickop{perf}
\quickop{tri}\quickop{Set}\quickop{Fib}\quickop{Nat}\quickop{haut}

\newcommand{\add}{\mathrm{add}}
\newcommand{\Fun}{\mathrm{Fun}} 

\newcommand{\bbL}{\mathbb{L}}
\newcommand{\Op}{\mathsf{op}}

\newcommand{\ie}{\textsl{i.e.}\ }

\newcommand{\too}{\longrightarrow}

\newcommand{\rep}{\mathrm{rep}} 

\numberwithin{equation}{section}

\newtheorem{theorem}[equation]{Theorem}
\newtheorem*{theorem*}{Theorem}

\newtheorem{lemma}[equation]{Lemma}
\newtheorem{proposition}[equation]{Proposition}
\theoremstyle{definition}
\newtheorem{definition}[equation]{Definition}

\newtheorem{example}[equation]{Example}
\newtheorem{notation}[equation]{Notation}
\newtheorem{convention}[equation]{Convention}

\xyoption{arrow}
\xyoption{matrix}
\xyoption{cmtip}
\SelectTips{cm}{}

\newdir{ >}{{}*!/-5pt/\dir{>}}

\begin{document}

\title[A characterization of the Chern character maps]{A universal characterization of the \\Chern character maps}
\author{Gon{\c c}alo~Tabuada}
\address{Departamento de Matem{\'a}tica e CMA, FCT-UNL, Quinta da Torre, 2829-516 Caparica,~Portugal}
\email{tabuada@fct.unl.pt}
\thanks{The author was partially supported by the Estimulo {\`a} Investiga{\c c}{\~a}o Award 2008 - Calouste Gulbenkian Foundation.}



\begin{abstract}
The Chern character maps are one of the most important working tools in mathematics. Although they admit numerous different constructions, they are not yet fully understood at the conceptual level. In this note we eliminate this gap by characterizing the Chern character maps, from the Grothendieck group to the (negative) cyclic homology groups, in terms of simple universal properties.
\end{abstract}

\maketitle
\setcounter{tocdepth}{0}
\section{introduction}
In his foundational work, Chern~\cite{Chern} introduced {\em character maps} from $K$-theory to de Rham cohomology in order to study complex vector bundles on smooth manifolds. Chern's construction made use of connection forms, exterior derivatives,~$\ldots$

Forty years latter, Connes~\cite{Connes} extended Chern's work to the non-commutative setting. He invented a new theory that plays the role of de Rham cohomology and constructed Chern character maps with values in it. Given a commutative base ring $k$ and a unital (but not necessarily commutative) $k$-algebra $A$, the (functorial) Chern characters maps 
\begin{eqnarray}\label{eq:chern-alg}
ch^-: K_0(A) \to HC_0^-(A) & & ch_n: K_0(A) \to HC_{2n}(A)\,, \,\, n \geq 0
\end{eqnarray}
go form the Grothendieck group to the (negative) cyclic homology groups; see \cite[\S8]{Loday}. Connes's construction made use of idempotents, a generalized trace map, $\ldots$

Latter, in the nineties, McCarthy~\cite{Mac} and Keller~\cite{Keller} extended\footnote{McCarthy worked in the context of exact $k$-categories. Keller then improved McCarthy's definition of cyclic homology, making it available for dg categories.} the Chern character maps (\ref{eq:chern-alg}) from $k$-algebras to dg categories; see \S\ref{sec:additive} for the notion of dg category. Given a commutative and unital base ring $k$, we have natural transformations 
\begin{eqnarray}\label{eq:cherns}
ch^-: K_0 \Rightarrow HC_0^-&&
ch_n: K_0 \Rightarrow HC_{2n}\,,\,\, n \geq 0
\end{eqnarray}
between functors defined on the category of small dg categories and with values in abelian groups; see \cite[\S4.4]{Mac}\cite{Keller}. The construction of the Chern character maps (\ref{eq:cherns}) made use of simplicial $S^1$-fixed points, a generalized $S_{\bullet}$-construction, $\ldots$

At this point one wonders if it is possible to isolate the fundamental properties of the Chern character maps while discarding specific constructions without loosing any information. The natural question is the following\,:

\vspace{0.2cm}

{\it How to conceptually understand/characterize the Chern character maps (\ref{eq:cherns}) ?}

\vspace{0.2cm}

In this note we provide an absolute and a relative solution to this question. Recall form \S\ref{sec:canonical} the description of the canonical isomorphisms 
\begin{eqnarray*}
\psi^-: HC_0^-(k) \stackrel{\sim}{\too} k && \psi_n: HC_{2n}(k) \stackrel{\sim}{\too} k \,,\,\, n \geq 0\,.
\end{eqnarray*} 
\begin{theorem}[Absolute characterization]\label{thm:A} The canonical maps 
\begin{eqnarray}
\Nat(K_0,HC_0^-) \stackrel{\sim}{\too} k && \eta \mapsto \psi^-(\eta(\underline{k})([k])) \label{eq:isom1}\\
\Nat(K_0,HC_{2n}) \stackrel{\sim}{\too} k && \eta \mapsto \psi_n(\eta(\underline{k})([k]))\,,\,\,n \geq 0 \label{eq:isom2}
\end{eqnarray}
are isomorphisms of abelian groups.
Here, $\underline{k}$ stands for the dg category with a single object and with $k$ as the (dg) $k$-algebra of endomorphisms, $[k]$ stands for the class of $k$ (as a module over itself) in the Grothendieck group $K_0(\underline{k})=K_0(k)$, and $\Nat$ stands for the abelian group of natural transformations (with group structure given by objectwise addition).

Under the canonical isomorphisms (\ref{eq:isom1}) and (\ref{eq:isom2}), the Chern character maps $ch^-$ and $ch_n$ (\ref{eq:cherns}) are characterized as the natural transformations corresponding to the unit ${\bf 1}$ of the base ring $k$.
\end{theorem}
Recall from \cite[\S5.1.8]{Loday} the sequence of natural transformations
\begin{equation}\label{eq:seq}
HC_0^- \stackrel{I}{\Rightarrow}HC_0^{\mathrm{per}} \Rightarrow \varprojlim_n HC_{2n} \Rightarrow \cdots  \stackrel{S}{\Rightarrow} HC_{2n} \stackrel{S}{\Rightarrow} HC_{2n-2} \Rightarrow \cdots \stackrel{S}{\Rightarrow} HC_0\,.
\end{equation}
\begin{theorem}[Relative characterization]\label{thm:B}
Fix a Chern character map $ch_m$ (\ref{eq:cherns}).
\begin{itemize}
\item[(i)] Then, the Chern character map $ch_n\,, n > m$, is the unique natural transformation such that $S^{(n-m)} \circ ch_n = ch_m$. 
\item[(ii)] Similarly, the Chern character map $ch^-$ is the unique natural transformation such that the composition
\begin{equation}\label{eq:comp}
 K_0 \stackrel{ch^-}{\Rightarrow} HC_0^- \stackrel{I}{\Rightarrow} HC_0^{\mathrm{per}} \Rightarrow \varprojlim_n HC_{2n} \Rightarrow \cdots  \stackrel{S}{\Rightarrow} HC_{2m}
\end{equation}
equals the Chern character map $ch_m$.
\end{itemize}
\end{theorem}
To the best of the author's knowledge, Theorems~\ref{thm:A} and \ref{thm:B} offer the first conceptual characterization of the Chern character maps. Theorem~\ref{thm:A} characterizes the Chern character maps among all natural transformations and Theorem~\ref{thm:B} asserts that, up to compatibility with the sequence (\ref{eq:seq}), there is only one Chern character map. It is expected that the conceptual understanding here presented will play a catalytic role in the several branches of mathematics where the Chern character maps are used. This will be the subject of future research. 
\begin{convention}
In the sequel $k$ denotes a commutative base ring with unit {\bf 1}.
\end{convention}
\section{Additive invariants of dg categories}\label{sec:additive}
In this section we review and develop some of the ingredients of the theory of additive invariants of dg categories \cite{IMRN, IMRNC} which will be used in the proofs of Theorems~\ref{thm:A} and \ref{thm:B}.

A {\em differential graded (=dg) category} (over the base ring $k$) is a category enriched over complexes of $k$-modules (Hom-sets are complexes) in such a way that the composition fulfills the graded Leibniz rule
$$d(f\circ g)=(df)\circ g+(-1)^{\textrm{deg}(f)}f\circ(dg)\,.$$
In particular, a dg category with a single object is simply a dg $k$-algebra. For a survey article on dg categories we invite the reader to consult Keller's ICM adress~\cite{ICM}. We denote by $\dgcat$ the category of small dg categories.

As in the case of a (dg) $k$-algebra, given a dg category $\cA$ we can consider its derived category $\cD(\cA)$ of right $\cA$-modules; see \cite[\S3.1]{ICM}. A dg functor $\cA \to \cB$ is called a {\em Morita equivalence} if the restriction of scalars functor $\cD(\cB) \stackrel{\sim}{\to} \cD(\cA)$ is an equivalence of (triangulated) categories; see \cite[\S3.8]{ICM}.

Let $\cA$ be a dg category. Consider the dg category $T(\cA)$ whose objects are the pairs $(i,x)$, where $i$ is an element of the set $\{1,2\}$ and $x$ is an object of $\cA$. The complex of morphisms in $T(\cA)$, from $(i,x)$ to $(i',x')$, is given by $\cA(x,x')$ if $i'\geq i$ and is $0$ otherwise. Composition is induced by the composition operation in $\cA$; see \cite[\S4]{IMRN} for details. Note that we have two natural inclusion dg functors
\begin{eqnarray*}
i_1: \cA \too T(\cA) && i_2: \cA \too T(\cA)\,.
\end{eqnarray*} 
\begin{definition}\label{def:additive}
Let $E:\dgcat \to \mathsf{D}$ be a functor with values in an additive category. We say that $E$ is an {\em additive invariant} if it satisfies the following two conditions\,:
\begin{itemize}
\item[(i)] it sends the Morita equivalences to isomorphisms;
\item[(ii)] given any dg category $\cA$, the inclusion dg functors induce an isomorphism\footnote{Condition (ii) can be equivalently formulated in terms of a general semi-orthogonal decomposition in the sense of Bondal-Orlov; see \cite[Thm.~6.3(4)]{IMRN}.}
$$ [E(i_1)\,\,\, E(i_2)]: E(\cA) \oplus E(\cA) \stackrel{\sim}{\too} E(T(\cA))\,.$$
\end{itemize}
A {\em morphism} of additive invariants is a natural transformation of functors. We denote by $\Fun_{\mathsf{A}}(\dgcat, \mathsf{D})$ the category of additive invariants with values in $\mathsf{D}$. 
\end{definition}
\begin{example} Examples of additive invariants with values in (the additive category of) abelian groups include the Grothendieck group functor ($K_0$), the cyclic homology group functors ($HC_n, n \geq 0$), the zero periodic homology group functor ($HC_0^{\mathrm{per}}$), the zero negative homology group functor ($HC_0^-$), $\ldots$; see \cite[\S6.1-6.2]{IMRN}.
\end{example}
As in the case of (dg) $k$-algebras, given two dg categories $\cA$ and $\cB$, we can form its tensor product $\cA \otimes  \cB$ (and also its derived version $\cA \otimes^{\bbL}\cB$); see \cite[\S2.3]{ICM}. Consider the additive category $\Hmo$ whose objects are the small dg categories and whose abelian groups of morphisms are given by
$$ \Hmo(\cA, \cB):= K_0 \,\rep(\cA,\cB)\,.$$
Here, $K_0$ is the Grothendieck group~\cite[Def.~4.5.8]{Neeman} of the full triangulated subcategory $\rep(\cA,\cB)$ of $\cD(\cA^\Op\otimes^{\bbL}\cB)$ whose objects are the bimodules $X$ such that for every object $x$ in $\cA$ the $\cB$-module $X(-,x)$ is compact in $\cD(\cB)$. Composition in $\Hmo$ is induced by the (derived) tensor product of bimodules; see \cite[\S6]{IMRN} for details. Note that we have a natural functor
\begin{equation}\label{eq:additive}
\cU_{\mathsf{A}}: \dgcat \too \Hmo
\end{equation}
which is the identity on objects and maps a dg functor $F:\cA \to \cB$ to the class in the Grothendieck group $K_0\, \rep(\cA,\cB)$ of the bimodule in $\rep(\cA,\cB)$ naturally associated to $F$. Since the functor $\cU_{\mathsf{A}}$ is the identity on objects, we will write, whenever it is clear front the context, $\cA$ instead of $\cU_{\mathsf{A}}(\cA)$.
\begin{theorem}\label{thm:additive}
The functor (\ref{eq:additive}) is the universal additive invariant, \ie given any additive category $\mathsf{D}$, we have an induced equivalence of categories
$$ (\cU_{\mathsf{A}})^{\ast}: \Fun_{\add}(\Hmo, \mathsf{D}) \stackrel{\sim}{\too} \Fun_{\mathsf{A}}(\dgcat, \mathsf{D})\,,$$
where $\Fun_{\add}(\Hmo, \mathsf{D})$ denotes the category of additive functors.
\end{theorem}
\begin{proof}
Let $E$ be an object of $\Fun_{\mathsf{A}}(\dgcat, \mathsf{D})$, \ie an additive invariant with values in $\mathsf{D}$. Thanks to \cite[Thms.~5.3 and 6.3]{IMRN} the functor $E$ factors uniquely through $\cU_{\mathsf{A}}$, giving rise to an additive functor $\overline{E}:\Hmo \to \mathsf{D}$. If $\eta$ is an element of $\Nat(E,E')$, \ie a morphism $\eta:E \Rightarrow E'$ of additive invariants, then $\overline{\eta}: \overline{E} \Rightarrow \overline{E'}$, with $\overline{\eta}(\cA):= \eta(\cA)$ for every dg category $\cA$, is a natural transformation of additive functors. We obtain then a well-defined functor
\begin{equation}\label{eq:inverse}
\overline{(-)}: \Fun_{\mathsf{A}}(\dgcat, \mathsf{D}) \too \Fun_{\add}(\Hmo,\mathsf{D})\,.
\end{equation}
Making use of \cite[Thms.~5.3 and 6.3]{IMRN}, we observe that the functors $(\cU_{\mathsf{A}})^{\ast}$ and $\overline{(-)}$ are (quasi-)inverse of each other. This achieves the proof.
\end{proof}
\begin{notation}
We will denote by $\underline{k}$ the dg category with a single object and with $k$ as the (dg) $k$-algebra of endomorphisms.
\end{notation}
\begin{lemma}{(\cite[Lem.~6.5]{IMRN})}\label{lem:additive}
Given any dg category $\cA$, we have a natural isomorphism of abelian groups
$$\Hom_{\Hmo}(\underline{k}, \cA) \simeq K_0(\cA)\,.$$
\end{lemma}
\section{Canonical Isomorphisms}\label{sec:canonical}
In this section we describe the canonical isomorphisms
\begin{eqnarray*}
\psi^-: HC_0^-(k) \stackrel{\sim}{\too} k && \psi_n: HC_{2n}(k) \stackrel{\sim}{\too} k \,,\,\, n \geq 0\,.
\end{eqnarray*} 
Straightforward computations show that in the cyclic bicomplex $CC(k)$~\cite[\S2.1.2]{Loday} the maps
\begin{eqnarray*}
b, -b': k^{\otimes(n+1)} \too k^{\otimes n}, \,n \geq 1&& (1-t), N: k^{\otimes(n+1)} \too k^{\otimes(n+1)},\, n \geq 0
\end{eqnarray*}
are given as follows (where $k^{\otimes m}$ was naturally identified with $k$)\,:
\begin{eqnarray*}
b = \left\{ \begin{array}{lcl}
0 & \text{for} & n~\mbox{odd} \\
\id & \text{for} & n~\mbox{even} \\
\end{array} \right.&& 
-b' = \left\{ \begin{array}{lcl}
-\id & \text{for} & n~\mbox{odd} \\
0 & \text{for} & n~\mbox{even} \\
\end{array} \right.
\end{eqnarray*}

\begin{eqnarray*}
(1-t) = \left\{ \begin{array}{lcl}
2 \id & \text{for} & n~\mbox{odd} \\
0 & \text{for} & n~\mbox{even} \\
\end{array} \right.&& 
N = \left\{ \begin{array}{lcl}
0 & \text{for} & n~\mbox{odd} \\
(n+1) \id & \text{for} & n~\mbox{even} \\
\end{array} \right.
\end{eqnarray*}
Therefore, in the total complex $\mathrm{Tot}\,CC(k)$ the differential map $$d:\mathrm{Tot}\,CC(k)_{m+1} \to \mathrm{Tot}\,CC(k)_m$$ vanishes for $m$ even and is surjective for $m$ odd. This implies that for $n \geq 0$, we have
\begin{eqnarray*}
HC_{2n+1}(k)=0 & \mathrm{and} & HC_{2n}(k)=Z_{2n}(\mathrm{Tot}\,CC(k))\,.
\end{eqnarray*}
A careful calculation shows that $Z_{2n}(\mathrm{Tot}\,CC(k))$ is the free $k$-module of rank one generated by the canonical cycle
\begin{equation}\label{eq:generator}
u^n:= (y_n, z_n, \ldots, y_i, z_i, \ldots, y_1, z_1, y_0) \in (\mathrm{Tot}\,CC(k))_{2n}
\end{equation}
where\footnote{Note that $\dfrac{(2i)!}{i!}$ and $\dfrac{(2i)!}{2(i!)}$ are integers.}
\begin{eqnarray*}
y_i:= (-1)^i \dfrac{(2i)!}{i!} {\bf 1} \,,\,\, 0 \leq i \leq n &\mathrm{and}& z_i:= (-1)^{i-1} \dfrac{(2i)!}{2(i!)} {\bf 1} \,,\,\, 1 \leq i \leq n \,.
\end{eqnarray*}
The isomorphisms $\psi_n, n \geq 0,$ are then the unique $k$-linear maps that send the canonical generator $u^n$ to the unit ${\bf 1}$ of the base ring $k$.

Now, recall from \cite[Prop.~5.1.9]{Loday} Milnor's short exact sequence
$$0\too {\varprojlim_n}^1 HC_{2n+1}(k) \too HC_0^{\mathrm{per}}(k) \too  {\varprojlim_n}\, HC_{2n}(k) \too 0\,.$$
Since $HC_{2n+1}(k)=0$ for $n \geq 0$, the Mittag-Leffler condition~\cite[\S5.1.10]{Loday} is automatically fulfilled and so the ${\lim}^1$-term vanishes. We obtain then a canonical isomorphism
\begin{equation}\label{eq:iso-per}
HC_0^{\mathrm{per}}(k) \stackrel{\sim}{\too} {\varprojlim_n} \,HC_{2n}(k)\,.
\end{equation}
Moreover, the periodicity map $S$~\cite[\S2.2]{Loday}, which is induced from the truncation of the cyclic bicomplex, sends the canonical generator $u^n$ of $HC_{2n}(k)$ to the canonical generator $u^{n-1}$ of $HC_{2n-2}(k)$. These facts imply that $HC_0^-(k)$ is the free $k$-module of rank one generated by the canonical cyle
\begin{equation}\label{eq:generator1}
u^{\infty}:=\varprojlim_n u^n = (\ldots, y_i, z_i, \ldots, y_1, z_1, y_0) \in (\mathrm{ToT}\,CC^-(k))_0\,.
\end{equation}
The isomorphism $\psi^-$ is then the unique $k$-linear map that send the canonical generator $u^{\infty}$ to the unit ${\bf 1}$ of the base ring $k$.
\section{Absolute characterization}
\begin{proposition}\label{prop:Yoneda}
Let $E: \dgcat \to \mathsf{Ab}$ be an additive invariant with values in (the additive category of) abelian groups. Then the canonical map
\begin{eqnarray}\label{eq:key}
\Nat(K_0,E) \stackrel{\sim}{\too} E(\underline{k}) && \eta \mapsto \eta(\underline{k})([k])
\end{eqnarray}
is an isomorphism of abelian groups. Here, $[k]$ stands for the class of $k$ (as a module over itself) in the Grothendieck group $K_0(\underline{k})=K_0(k)$ and $\Nat$ stands for the abelian group of natural transformations (with group structure given by objectwise addition).
\end{proposition}
\begin{proof}
The functors $K_0$ and $E$ are additive invariants and so they belong to the category $\Fun_{\mathsf{A}}(\dgcat, \mathsf{Ab})$. Using the functor (\ref{eq:inverse}) we obtain then an isomorphism
\begin{eqnarray}\label{number:A}
\Nat(K_0,E) \stackrel{\sim}{\too} \Nat (\overline{K_0}, \overline{E}) && \eta \mapsto \overline{\eta}\,,
\end{eqnarray}
which is moreover compatible with the group structure (on both sides) given by objectwise addition. Thanks to Lemma~\ref{lem:additive} the additive functor $\overline{K_0}$ is corepresentable in $\Hmo$ by the object $\underline{k}$. Therefore, since every additive functor is a $\mathsf{Ab}$-functor~\cite[Def.~6.2.3]{Borceaux}, the enriched Yoneda Lemma~\cite[Thm.~8.3.5]{Borceaux} furnish us an isomorphism of abelian groups
\begin{eqnarray}\label{number:B}
\Nat(\overline{K_0},\overline{E}) \stackrel{\sim}{\too} \overline{E}(\underline{k}) & & \overline{\eta} \mapsto \overline{\eta}(\underline{k})(\id_{\underline{k}})\,.
\end{eqnarray}
Note that by construction of $\Hmo$, the identity $\id_{\underline{k}}$ of the object $\underline{k}$ in $\Hmo$ corresponds to the class $[k]$ of $k$ (as a module over itself) in the Grothendieck group $K_0(\underline{k})=K_0(k)$. Since $\overline{E}(\underline{k})=E(\underline{k})$ and $\overline{\eta}(\underline{k})=\eta(\underline{k})$, we conclude that the canonical map (\ref{eq:key}) is the composition of the isomorphisms (\ref{number:A}) and (\ref{number:B}). This achieves the proof.
\end{proof}
\begin{proof}[{\bf Proof of Theorem~\ref{thm:A}}]
We have the equalities 
\begin{eqnarray*}
HC_0^-(\underline{k})=HC_0^-(k)& \mathrm{and} & HC_{2n}(\underline{k})=HC_{2n}(k)\,,\,\,n \geq 0
\end{eqnarray*}
and the functors $HC_0^-$ and $HC_{2n}, n \geq 0,$ are additive invariants with values in abelian groups. Therefore, the canonical isomorphisms (\ref{eq:isom1}) and (\ref{eq:isom2}) follow from combining Proposition~\ref{prop:Yoneda} with the isomorphisms $\psi^-$ and $\psi_n$.

Now, let $ch_n$ be a Chern character map. Thanks to the agreement property~\cite[\S4.5]{Mac}, when we evaluate $ch_n$ at $\underline{k}$ we obtain the map \begin{equation}\label{eq:map1}
K_0(k) \too HC_{2n}(k)
\end{equation}
of \cite[Thm.~8.3.4]{Loday} (with $A=k$). Note that in order to calculate the value $ch_n(\underline{k})([k])$ of the map (\ref{eq:map1}) at $[k]$ we can choose for idempotent $e$ the unit ${\bf 1}$ of the base ring $k\simeq \cM_1(k)$; see \cite[\S8.3.1]{Loday}. Therefore, a careful analysis shows that $ch_n(\underline{k})([k])$ is precisely the canonical generator $u^n$ (\ref{eq:generator}) of $HC_{2n}(k)$; see \cite[Lem.~8.3.3]{Loday}. Using $\psi_n$, we conclude that under the canonical isomorphism (\ref{eq:isom2}), the Chern character map $ch_n$ is characterized as the natural transformation corresponding to the unit ${\bf 1}$ of the base ring $k$. 

We now consider the Chern character map $ch^-$. Thanks to the agreement property~\cite[\S4.5]{Mac}, when we evaluate $ch^-$ at $\underline{k}$ we obtain the map
\begin{equation}\label{eq:map2}
K_0(k) \too HC_0^-(k)
\end{equation}
of \cite[Prop.~8.3.8]{Loday} (with $A=k$). Once again, in order to calculate the value $ch^-(\underline{k})([k])$ of the map (\ref{eq:map2}) at $[k]$, we can choose for idempotent $e$ the unit ${\bf 1}$ of the base ring $k\simeq \cM_1(k)$. In this case we realize that $ch^-(\underline{k})([k])$ is precisely the canonical generator $u^{\infty}$ (\ref{eq:generator1}) of $HC_0^-(k)$. Using $\psi^-$, we conclude that under the canonical isomorphism (\ref{eq:isom1}), the Chern character map $ch^-$ is characterized as the natural transformation corresponding to the unit ${\bf 1}$ of the base ring $k$.
\end{proof} 
\section{Relative characterization}
\begin{proof}[{\bf Proof of Theorem~\ref{thm:B}}]
By construction, see \cite[Defs.~4.4.1-4.4.2]{Mac} and diagram \cite[(11.4.3.3)]{Loday} (with $n=0$), we have $S^{(n-m)}\circ ch_n = ch_m$. Now, recall that the periodicity map
$$ S: HC_{2n} \Rightarrow HC_{2n-2}\,,$$
when evaluated at $\underline{k}$, sends the canonical generator $u^n$ (\ref{eq:generator}) of $HC_{2n}(k)=HC_{2n}(\underline{k})$ to the canonical generator $u^{n-1}$ of $HC_{2n-2}(k)$; see \cite[Rk.~2.2.2]{Loday}. We have then the following commutative diagram
\begin{equation}\label{eq:diag}
\xymatrix{
 HC_{2n}(k) \ar[d]_{\sim}^{\psi_n} \ar[r]^-S & HC_{2n-2}(k) \ar[d]_{\sim}^{\psi_{(n-1)}} \ar[r]^-S & \cdots \ar@{}[d]|{\cdots} \ar[r]^-S & HC_{2m}(k)  \ar[d]_{\sim}^{\psi_m} \\ k \ar@{=}[r] & k \ar@{=}[r] & \cdots \ar@{=}[r] & k
}
\end{equation}
Thanks to Theorem~\ref{thm:A} the Chern character map $ch_n$ corresponds, under the canonical isomorphism (\ref{eq:isom2}), to the unit ${\bf 1}$ of the base ring $k$. Therefore, using the above commutative diagram (\ref{eq:diag}), we conclude that $ch_n$ is in fact the unique natural transformation such that $S^{(n-m)}\circ ch_n = ch_m$. This shows item (i). 

Let us now show item (ii). Once again by construction, the composition (\ref{eq:comp}) equals the Chern character map $ch_m$; see~\cite[Prop.~8.3.8]{Loday}. Recall from~\cite[\S5.1]{Loday} that the natural transformation $$I: HC_0^- \Rightarrow HC_0^{\mathrm{per}}$$
is the identity. We have then the following commutative diagram
\begin{equation}\label{eq:all}
\xymatrix{
HC_0^-(k) \ar@{=}[r] \ar[d]_{\sim}^{\psi^-} & HC_0^{\mathrm{per}}(k) \ar[d]_{\sim}^{\psi^-} \ar[r]_-{\sim}^-{\theta} & \varprojlim HC_{2n}(k) \ar[d]^{\varprojlim \psi_n}_{\sim} \ar[r] & \cdots \ar@{}[d]|{\cdots} \ar[r]^-S & HC_{2m}(k) \ar[d]_{\sim}^{\psi_m} \\
k \ar@{=}[r] & k \ar@{=}[r] & k \ar@{=}[r] & \cdots \ar@{=}[r] & k \,,
}
\end{equation}
where $\theta$ is the isomorphism (\ref{eq:iso-per}).
Thanks to Theorem~\ref{thm:A} the Chern character map $ch^-$ corresponds, under the canonical isomorphism (\ref{eq:isom1}), to the unit ${\bf 1}$ of the base ring $k$. Therefore, using Lemma~\ref{lem:lim}, Proposition~\ref{prop:Yoneda}, and the above commutative diagram (\ref{eq:all}), we conclude that $ch^-$ is in fact the unique natural transformation such that the composition (\ref{eq:comp}) equals the Chern character map $ch_m$. This shows item (ii).
\end{proof}
\begin{lemma}\label{lem:lim}
The functor
\begin{equation*}
 \varprojlim_n HC_{2n}: \dgcat \too \mathsf{Ab} 
 \end{equation*}
is an additive invariant.
\end{lemma}
\begin{proof}
Condition (i) is clear. Condition (ii) follows from the fact that the functors $HC_{2n}, n \geq 0,$ are additive invariants and that finite sums agree with finite products in $\mathsf{Ab}$. 
\end{proof} 
\noindent\textbf{Acknowledgments:} The author is very grateful to Paul Balmer, Andrew Blumberg, David Gepner and Christian Haesemeyer for stimulating conversations. He would also like to thank the Department of Mathematics at UCLA for its hospitality and excellent working conditions, where part of this work was carried out.


\begin{thebibliography}{00}

\bibitem{Borceaux} F.~Borceux, {\em Handbook of categorical algebra. 2}. Encyclopedia of Mathematics and its Applications {\bf 51} (1994). Cambridge Univ. Press.

\bibitem{Chern} S.~S.~Chern, {\em Characteristic classes of Hermitian Manifolds}. Ann.~Math. {\bf 47}(1) (1946), 85--121.

\bibitem{Connes} A.~Connes, {\em Non-commutative differential geometry}. Publ.~Math. de l'IH{\'E}S {\bf 62} (1985), 41--144.
 
\bibitem{ICM} B.~Keller, {\em On differential graded categories}. International Congress of Mathematicians (Madrid), Vol.~{\bf II} (2006), 151--190. Eur.~Math.~Soc., Z{\"u}rich.

\bibitem{Keller} \bysame, {\em On the cyclic homology of exact
    categories}. J.~Pure~Appl.~Algebra {\bf 136}(1) (1999), 1--56.

\bibitem{Loday} J.-L.~Loday, {\em Cyclic homology}. Grundlehren der Mathematischen Wissenschaften {\bf 301} (1992). Springer-Verlag, Berlin.

\bibitem{Mac} R. McCarthy, {\em The cyclic homology of an exact category}. J.~Pure~Appl.~Algebra {\bf 93}(3) (1994), 251--296.
  
\bibitem{Neeman} A.~Neeman, {\em Triangulated categories}. Ann.~Math.~Studies {\bf 148} (2001). Princeton Univ. Press.

\bibitem{IMRN} G.~Tabuada, {\em Invariants additifs de dg-cat{\'e}gories}.
Int.~Math.~Res.~Not. {\bf 53} (2005), 3309--3339.

\bibitem{IMRNC} \bysame, {\em Corrections {\`a} Invariants additifs de dg-cat{\'e}goriesÕ}, Int.~Math.~Res.~Not., article ID rnm {\bf 149} (2007).

\end{thebibliography}
\end{document}